\documentclass[12pt]{article}
\usepackage{amsmath,amssymb,amsthm}
\numberwithin{equation}{section}
\usepackage{units}
\usepackage{color}
\usepackage[T1]{fontenc}
\usepackage[utf8]{inputenc}
\usepackage{authblk}
\usepackage{bm}
\usepackage{graphicx}

\usepackage{datetime}

\usepackage[colorlinks=true,
linkcolor=webgreen,
filecolor=webbrown,
citecolor=webgreen]{hyperref}

\definecolor{webgreen}{rgb}{0,.5,0}
\definecolor{webbrown}{rgb}{.6,0,0}

\usepackage[toc,page]{appendix}

\newcommand{\blue}[1]{{\color{blue}#1}}
\newcommand{\red}[1]{{\color{red}#1}}

\newtheorem{thm}{Theorem}
\newtheorem{theorem}[thm]{Theorem}
\newtheorem{lemma}{Lemma}

\title{Multiple sums involving the terms of a general second order sequence of numbers}

\author[]{Kunle Adegoke \\\href{mailto:adegoke00@gmail.com}{\tt adegoke00gmail.com}}

\affil{Department of Physics and Engineering Physics, \mbox{Obafemi Awolowo University}, 220005 Ile-Ife, Nigeria}

\begin{document}
\date{}

\maketitle

\begin{abstract}
\noindent We evaluate the nested sum $\sum_{a_{n - 1}  = c}^{a_n } {\sum_{a_{n - 2}  = c}^{a_{n - 1} } { \cdots \sum_{a_0  = c}^{a_1 } {x^{a_0 } } } }$ where $a_n$ and $c$ are any integers and $x$ is a real or complex variable. Consequently, we evaluate multiple sums involving the terms of a general second order sequence, the Horadam sequence $(W_j(a,b;p,q))$, defined for all non-negative integers $j$ by the recurrence relation $W_0  = a,\,W_1  = b;\,W_j  = pW_{j - 1}  - qW_{j - 2}\, (j \ge 2)$; where $a$, $b$, $p$ and $q$ are arbitrary complex numbers, with $p\ne 0$, $q\ne 0$.

\end{abstract}
\noindent 2010 {\it Mathematics Subject Classification}:
Primary 11B39; Secondary 11B37.

\noindent \emph{Keywords: }
Horadam sequence, Fibonacci number, Lucas number, Lucas sequence, summation identity, nested sum, multiple sum.

\section{Introduction}
Let $F_j$ be the $j^{th}$ Fibonacci number. Ivie~\cite{ivie69} has shown that
\[
\sum_{s = 1}^m {\sum_{r = 1}^s {F_r } }  = F_{m + 4}  - F_4  - m,
\]

\[
\sum_{m = 1}^n {\sum_{s = 1}^m {\sum_{r = 1}^s {F_r } } }  = F_{n + 6}  - F_6  - nF_4  - \frac{{n(n + 1)}}{2},
\]
and more generally,
\begin{equation}\tag{H}\label{eq.b1al0dl}
\sum_{a_{n - 1}  = 1}^{a_n } {\sum_{a_{n - 2}  = 1}^{a_{n - 1} } { \cdots \sum_{a_0  = 1}^{a_1 } {F_{a_0 } } } }  = F_{a_n  + 2n}  - \sum_{j = 0}^{n - 1} {F_{2(n - j)} \binom{a_n + j - 1}j} .
\end{equation}
In this paper we will extend the study to the Horadam sequence and derive more such sums. 

Identity~\eqref{eq.b1al0dl} is a special case of the more general identity (see Theorem~\ref{thm.mrlteua})
\begin{equation*}
\begin{split}
&\sum_{a_{n - 1}  = c}^{a_n } {\sum_{a_{n - 2}  = c}^{a_{n - 1} } { \cdots \sum_{a_0  = c}^{a_1 } {\frac{{W_{ra_0  + s} }}{{V_r^{a_0 } }}} } }\\
&\qquad  = ( - 1)^n \frac{{W_{r(a_n  + 2n) + s} }}{{q^{rn} V_r^{a_n } }} - \frac1{V_r^{c - 1}}\sum_{j = 0}^{n - 1} {( - 1)^{n - j} \frac{{W_{r(2n - 2j + c - 1) + s} }}{{q^{r(n - j)} }}\binom{a_n + j - c}j},
\end{split}
\end{equation*}
in which  $(W_j(a,b;p,q))$ is the Horadam sequence~\cite{horadam65} defined for all non-negative integers $j$, by the recurrence relation
\begin{equation}\label{eq.vhrb5b3}
W_0  = a,\,W_1  = b;\,W_j  = pW_{j - 1}  - qW_{j - 2}\, (j \ge 2);
\end{equation}
where $a$, $b$, $p$ and $q$ are arbitrary complex numbers, with $p\ne 0$, $q\ne 0$.

Two important cases of  $(W_j)$ are the Lucas sequences of the first kind, $(U_j(p,q))=(W_j(0,1;p,q))$, and of the second kind, $(V_j(p,q))=(W_j(2,p;p,q))$; so that 
\[
\mbox{$U_0=0$, $U_1=1$};\, U_j=pU_{j-1}-qU_{j-2}, \mbox{($j\ge 2$)};
\]
and
\[
\mbox{$V_0=2$, $V_1=p$};\, V_j=pV_{j-1}-qV_{j-2}, \mbox{($j\ge 2$)}.
\]
The most well-known Lucas sequences are the Fibonacci sequence, $(F_j)=(U_j(1,-1))$ and the sequence of Lucas numbers, $(L_j)=(V_j(1,-1))$.

The special case $p=1$ is also important, giving the sequence $(w_j(a,b;q))=(W_j(a,b;1,q))$, with corresponding special Lucas sequences $(u_j(q))=(U_j(1,q))$ and $(v_j(q))=(V_j(1,q))$. The particular case $(G_j(a,b))=(w_j(a,b;-1))$ is the so-called gibonacci sequence, with the Fibonacci and Lucas sequences, $(F_j)=(G_j(0,1))$ and $(L_j)=(G_j(2,1))$, as special cases. Explicitly,
\[
G_0  = a,\,G_1  = b;\,G_j  = G_{j - 1}  + G_{j - 2}\, (j \ge 2).
\]
The Binet formulas for sequences $(U_j)$, $(V_j)$ and $(W_j)$ in the non-degenerate case, $p^2 - 4q > 0$, are
\begin{equation}\tag{BW}\label{eq.mt3covv}
U_j=\frac{\tau^j-\sigma^j}{\tau-\sigma}=\frac{\tau^j-\sigma^j}\Delta,\qquad V_j=\tau^j+\sigma^j, \qquad
W_j = \mathbb A\tau ^j  + \mathbb B\sigma ^j\,,
\end{equation}
with
\[
\mathbb A=\frac{{b - a\sigma}}{{\tau  - \sigma }},\quad \mathbb B=\frac{{a\tau  - b}}{{\tau  - \sigma }},
\]
where 
\[\tau=\tau(p,q)=\frac{p+\sqrt{p^2-4q}}2,\quad\sigma=\sigma(p,q)=\frac{p-\sqrt{p^2-4q}}2,\]
are the distinct zeros of the characteristic polynomial $x^2-px+q$ of the Horadam sequence; so that $\tau\sigma=q$ and $\tau + \sigma=p$.

The Binet formulas for the Fibonacci and Lucas numbers are
\begin{equation}\tag{BF}\label{eq.qohepwk}
F_j  = \frac{{\alpha ^j  - \beta ^j }}{{\alpha  - \beta }} = \frac{{\alpha ^j  - \beta ^j }}{{\sqrt 5 }},\quad L_j  = \alpha ^j  + \beta ^j,
\end{equation}
where $\alpha=\tau(1,-1)=(1 + \sqrt 5)/2$ is the golden ratio and $\beta=\sigma(1,-1)=-1/\alpha$.

Extension of the definition of $W_n$ to negative subscripts is provided by writing the recurrence relation as $W_{-n}=(pW_{-n+1}-W_{-n+2})/q$.

\section{Preliminary results}
Let $x$ be a real or complex variable, $a_n$ an integer and $n$ a positive integer. We wish to evaluate
\[
\sum_{a_{n - 1}  = 1}^{a_n } {\sum_{a_{n - 2}  = 1}^{a_{n - 1} } { \cdots \sum_{a_0  = 1}^{a_1 } {x^{a_0 } } } };
\]
the basic ingredient in our derivations.

The geometric progression sum
\[
\sum_{k = 1}^m {x^k }  = \frac{{x^{m + 1}  - x}}{{x - 1}}
\]
can be arranged as
\begin{equation}\tag{D}\label{eq.p2fza8v}
\frac{{x - 1}}{x}\sum_{j = 1}^m {x^j }  = x^m  - 1.
\end{equation}

Multiply through the identity
\begin{equation}\label{eq.m05ob73}
\frac{{x - 1}}{x}\sum_{a_0  = 1}^{a_1 } {x^{a_0 } }  = x^{a_1 }  - 1
\end{equation}
by $(x-1)/x$ and sum over $a_1$, making use of \eqref{eq.p2fza8v} with $m=a_2$, to obtain
\[
\begin{split}
\left( {\frac{{x - 1}}{x}} \right)^2 \sum_{a_1  = 1}^{a_2 } {\sum_{a_0  = 1}^{a_1 } {x^{a_0 } } } & = \frac{{x - 1}}{x}\sum_{a_1  = 1}^{a_2 } {x^{a_1 } }  - \frac{{x - 1}}{x}\sum_{a_1  = 1}^{a_2 } 1\\
&= x^{a_2 }  - 1 - \left( {\frac{{x - 1}}{x}} \right)\sum_{a_1  = 1}^{a_2 } 1.
\end{split}
\]
Multiply through the above by $(x-1)/x$ and sum over $a_2$, making use of \eqref{eq.p2fza8v} with $m=a_3$. This gives
\[
\begin{split}
\left( {\frac{{x - 1}}{x}} \right)^3 \sum_{a_2 = 1}^{a_3 } {\sum_{a_1  = 1}^{a_2 } {\sum_{a_0  = 1}^{a_1 } {x^{a_0 } } } }  &= \frac{{x - 1}}{x}\sum_{a_2  = 1}^{a_3 } {x^{a_2 } }  - \frac{{x - 1}}{x}\sum_{a_2  = 1}^{a_3 } 1  - \left( {\frac{{x - 1}}{x}} \right)^2 \sum_{a_2  = 1}^{a_3 } {\sum_{a_1  = 1}^{a_2 } 1 } \\
&= x^{a_3 }  - 1 - \frac{{x - 1}}{x}\sum_{a_2  = 1}^{a_3 } 1  - \left( {\frac{{x - 1}}{x}} \right)^2 \sum_{a_2  = 1}^{a_3 } {\sum_{a_1  = 1}^{a_2 } 1 }.
\end{split}
\]
Continuing the iteration, we find
\begin{equation}\label{eq.v0zquxz}
\left( {\frac{{x - 1}}{x}} \right)^n \sum_{a_{n - 1}  = 1}^{a_n } {\sum_{a_{n - 2}  = 1}^{a_{n - 1} } { \cdots \sum_{a_0  = 1}^{a_1 } {x^{a_0 } } } }  = x^{a_n }  - 1 - \sum_{j = 1}^{n - 1} {\left( {\frac{{x - 1}}{x}} \right)^j \sum_{a_{n - 1}  = 1}^{a_n } {\sum_{a_{n - 2}  = 1}^{a_{n - 1} } { \cdots \sum_{a_{n - j}  = 1}^{a_{n - j + 1} } 1 } } }.
\end{equation}
Thus, the task of finding $\sum_{a_{n - 1}  = 1}^{a_n } {\sum_{a_{n - 2}  = 1}^{a_{n - 1} } { \cdots \sum_{a_0  = 1}^{a_1 } {x^{a_0 } } } }$ reduces to that of evaluating\\ $\sum_{a_{n - 1}  = 1}^{a_n } {\sum_{a_{n - 2}  = 1}^{a_{n - 1} } { \cdots \sum_{a_{n - j}  = 1}^{a_{n - j + 1} } 1 } }$ for the $n-1$ values of $j$. Note that there are exactly $j$ sums in the latter multiple sum. We will prove a lemma and return to \eqref{eq.v0zquxz}.
\begin{lemma}\label{lem.hy9navo}
Let $k$, $m$ and $b_s$ be non-negative integers and let $s$ be a positive integer. Then
\begin{equation}\label{eq.ckl3mw8}
\sum_{j=1}^m{\binom{j + k - 1}k}=\binom{k + m}{k + 1},
\end{equation}

\begin{equation}\label{eq.p3e2mj9}
\sum_{b_{s - 1}  = 1}^{b_s } {\sum_{b_{s - 2}  = 1}^{b_{s - 1} } { \cdots \sum_{b_0  = 1}^{b_1 } 1 } }  = \binom{b_s + s - 1}s.
\end{equation}
\end{lemma}
\begin{proof}
Both identities will be proved by induction. To prove \eqref{eq.ckl3mw8} we keep $k$ fixed and carry out an induction on $m$. Identity \eqref{eq.ckl3mw8} is obviously true for $m=0$ and $m=1$. Assume the truth for $m=r$ to establish the induction hypothesis:
\[
P_r:\quad \sum_{j=1}^r{\binom{j + k - 1}k}=\binom{k + r}{k + 1}.
\]
We wish to prove that $P_r\implies P_{r + 1}$ for $r$ a positive integer.
\[
\begin{split}
\sum_{j=1}^{r + 1}{\binom{j + k - 1}k}&=\sum_{j=1}^r{\binom{j + k - 1}k} + \binom{k + r}k\\
&=\binom{k + r}{k + 1} + \binom{k + r}k, \text{ by hypothesis $P_r$},\\
&=\binom{k + r + 1}{k + 1},
\end{split}
\]
where, in the last step, we used Pascal's identity:
\[
\binom s{k + 1} + \binom sk=\binom {s + 1}{k + 1}.
\]
Thus, $P_r\implies P_{r + 1}$.

Since
\[
\sum_{b_0=1}^{b_1}1=\binom{b_1}1=b_1,
\]
identity \eqref{eq.p3e2mj9} holds for $s=1$. Assume the veracity for $s=k$ to get the induction hypothesis:
\[
P_k:\quad\sum_{b_{k - 1}  = 1}^{b_k } {\sum_{b_{k - 2}  = 1}^{b_{k - 1} } { \cdots \sum_{b_0  = 1}^{b_1 } 1 } }  = \binom{b_k + k - 1}k.
\]
We have
\[
\begin{split}
\sum_{b_{k}  = 1}^{b_{k + 1} } {\sum_{b_{k - 1}  = 1}^{b_k }\sum_{b_{k - 2}  = 1}^{b_{k - 1} } { \cdots \sum_{b_0  = 1}^{b_1 } 1 } }  &= \sum_{b_k=1}^{b_{k + 1}}{\binom{b_k + k - 1}k},\text{by hypothesis $P_k$},\\
&=\binom{b_{k+1} + k}{k + 1}, \text{by \eqref{eq.ckl3mw8} with $m=b_{k + 1}$}.
\end{split}
\]
Thus, $P_k\implies P_{k + 1}$.
\end{proof}
Using identity~\eqref{eq.p3e2mj9}, we find
\[
\sum_{a_{n - 1}  = 1}^{a_n } {\sum_{a_{n - 2}  = 1}^{a_{n - 1} } { \cdots \sum_{a_{n - j}  = 1}^{a_{n - j + 1} } 1 } }=\binom{a_n + j - 1}j,
\]
which inserted in \eqref{eq.v0zquxz}, yields
\begin{equation}\label{eq.y0z4ke0}
\left( {\frac{{x - 1}}{x}} \right)^n \sum_{a_{n - 1}  = 1}^{a_n } {\sum_{a_{n - 2}  = 1}^{a_{n - 1} } { \cdots \sum_{a_0  = 1}^{a_1 } {x^{a_0 } } } }  = x^{a_n }  - 1 - \sum_{j = 1}^{n - 1} {\left( {\frac{{x - 1}}{x}} \right)^j \binom{a_n + j - 1}j }.
\end{equation}
We now give a formal proof of \eqref{eq.y0z4ke0}. In section~\ref{sec.jmspw0p} we will consider some of its applications.
\begin{lemma}
Let $n$ be a positive integer. Let $x$ be a real or complex variable. Then,
\begin{equation}\tag{E}\label{eq.bmc15xg}
\left( {\frac{{x - 1}}{x}} \right)^n \sum_{a_{n - 1}  = 1}^{a_n } {\sum_{a_{n - 2}  = 1}^{a_{n - 1} } { \cdots \sum_{a_0  = 1}^{a_1 } {x^{a_0 } } } }  = x^{a_n } - \sum_{j = 0}^{n - 1} {\left( {\frac{{x - 1}}{x}} \right)^j \binom{a_n + j - 1}j }.
\end{equation}
\end{lemma}
\begin{proof}
The proof is by induction on $n$. The identity is readily verified to be true for $n=1$, giving \eqref{eq.m05ob73}. Assume the truth for $n=k$. We have the induction hypothesis:
\[
P_k:\quad\left( {\frac{{x - 1}}{x}} \right)^k \sum_{a_{k - 1}  = 1}^{a_k } {\sum_{a_{k - 2}  = 1}^{a_{k - 1} } { \cdots \sum_{a_0  = 1}^{a_1 } {x^{a_0 } } } }  = x^{a_k } - \sum_{j = 0}^{k - 1} {\left( {\frac{{x - 1}}{x}} \right)^j \binom{a_k + j - 1}j }.
\]
We wish to prove that $P_k\implies P_{k + 1}$.

We have
\begin{equation}\label{eq.gay98lw}
\begin{split}
&\left( {\frac{{x - 1}}{x}} \right)^{k + 1} \sum_{a_k  = 1}^{a_{k + 1} } {\sum_{a_{k - 1}  = 1}^{a_k } {\sum_{a_{k - 2}  = 1}^{a_{k - 1} } { \cdots \sum_{a_0  = 1}^{a_1 } {x^{a_0 } } } } }\\ 
&\qquad= \left( {\frac{{x - 1}}{x}} \right)\sum_{a_k  = 1}^{a_{k + 1} } {\left\{ {\left( {\frac{{x - 1}}{x}} \right)^k \sum_{a_{k - 1}  = 1}^{a_k } {\sum_{a_{k - 2}  = 1}^{a_{k - 1} } { \cdots \sum_{a_0  = 1}^{a_1 } {x^{a_0 } } } } } \right\}} \\
&\qquad= \left( {\frac{{x - 1}}{x}} \right)\sum_{a_k  = 1}^{a_{k + 1} } {\left\{ {x^{a_k }  - \sum_{j = 0}^{k - 1} {\left( {\frac{{x - 1}}{x}} \right)^j \binom{a_k + j - 1}j} } \right\}}, \text{by induction hypothesis $P_k$}, \\
&\qquad= \red{\frac{x - 1}x\sum_{a_k=1}^{a_{k + 1}}{x^{a_k}}} - \blue{\sum_{j = 0}^{k - 1} {\left( {\frac{{x - 1}}{x}} \right)^{j + 1} \sum_{a_k  = 1}^{a_{k + 1} } \binom{a_k + j - 1}j } }.
\end{split}
\end{equation}
Now,
\begin{equation}\label{eq.uvu2er5}
\red{\frac{{x - 1}}{x}\sum_{a_k  = 1}^{a_{k + 1} } {x^{a_k } }}  = \frac{{x - 1}}{x}\,\frac{{x^{a_{k + 1}  + 1}  - x}}{{x - 1}} = x^{a_{k + 1} }  - 1
\end{equation}
and
\begin{equation}\label{eq.ev55zio}
\begin{split}
\blue{\sum_{j = 0}^{k - 1} {\left( {\frac{{x - 1}}{x}} \right)^{j + 1} \sum_{a_k  = 1}^{a_{k + 1} } \binom{a_k + j - 1}j }}  &= \sum_{j = 0}^{k - 1} {\left( {\frac{{x - 1}}{x}} \right)^{j + 1} \binom{a_{k + 1} + j}{j + 1}},\text{by \eqref{eq.ckl3mw8}}, \\
&= \sum_{j = 1}^k {\left( {\frac{{x - 1}}{x}} \right)^j \binom{a_{k + 1} + j - 1}j}\\ 
&= \sum_{j = 0}^k {\left( {\frac{{x - 1}}{x}} \right)^j \binom{a_{k + 1} + j - 1}j}  - 1.
\end{split}
\end{equation}
Using \eqref{eq.uvu2er5} and \eqref{eq.ev55zio} in \eqref{eq.gay98lw}, we find
\[
P_{k + 1}:\quad\left( {\frac{{x - 1}}{x}} \right)^{k + 1} \sum_{a_k  = 1}^{a_{k + 1} } {\sum_{a_{k - 1}  = 1}^{a_k } {\sum_{a_{k - 2}  = 1}^{a_{k - 1} } { \cdots \sum_{a_0  = 1}^{a_1 } {x^{a_0 } } } } }=x^{a_{k + 1}} - \sum_{j = 0}^k {\left( {\frac{{x - 1}}{x}} \right)^j \binom{a_{k + 1} + j - 1}j}.
\]
Thus, $P_k\implies P_{k + 1}$.
\end{proof}
\section{Main results}\label{sec.jmspw0p}
In this section, we will apply identity \eqref{eq.bmc15xg} to derive some nested identities involving Horadam numbers.

First we wish to modify identity \eqref{eq.bmc15xg} so that each sum in the multiple sum on the left hand side starts with an arbitrary index, say $c$. To do this, we replace the sequence of integers $(a_i)$, $i=0,1,2,\ldots,n$ with $(a_i - c + 1)$, $i=0,1,2,\ldots,n$. This gives
\begin{equation}\label{eq.tms1dws}
\left( {\frac{{x - 1}}{x}} \right)^n \sum_{a_{n - 1}  = c}^{a_n } {\sum_{a_{n - 2}  = c}^{a_{n - 1} } { \cdots \sum_{a_0  = c}^{a_1 } {x^{a_0 } } } }  = x^{a_n }  - x^{c - 1}\sum_{j = 0}^{n - 1} {\left( {\frac{{x - 1}}{x}} \right)^j \binom{a_j + j - c}j}.
\end{equation}
Note that the identities in Lemma~\ref{lem.hy9navo} can also be caused to start from $j=c$ by writing $m - c + 1$ for $m$ and $j - c + 1$ for $j$, and replacing $(b_i)$, $i=0,1,2,\ldots,s$ with $(b_i - c + 1)$, $i=0,1,2,\ldots,s$, giving
\begin{equation}\label{eq.nwt0kv5}
\sum_{j=c}^m{\binom{j - c + k}k}=\binom{m - c + k + 1}{k + 1},
\end{equation}

\begin{equation}\label{eq.unc9ugp}
\sum_{b_{s - 1}  = c}^{b_s } {\sum_{b_{s - 2}  = c}^{b_{s - 1} } { \cdots \sum_{b_0  = c}^{b_1 } 1 } }  = \binom{b_s + s - c}s.
\end{equation}
The identity obtained by setting $c=0$ in \eqref{eq.unc9ugp} provided the motivation for Butler and Karasik~\cite{butler10} to look at nested sums.

Multiplying through \eqref{eq.tms1dws} by $(x/(x - 1))^n$ and writing $x/y$ for $x$ and $-x/y$ for $x$, in turn, we have the following useful versions:
\begin{equation}\tag{A}\label{eq.rjmjqrv}
\begin{split}
f(x,y;a_n,n,c)&=\sum_{a_{n - 1}  = c}^{a_n } {\sum_{a_{n - 2}  = c}^{a_{n - 1} } { \cdots \sum_{a_0  = c}^{a_1 } {\left( {\frac{x}{y}} \right)^{a_0 } } } }\\
&  = \left( {\frac{x}{{x - y}}} \right)^n \left( {\frac{x}{y}} \right)^{a_n }  - \sum_{j = 0}^{n - 1} {\left( {\frac{x}{{x - y}}} \right)^{n - j}\left(\frac xy\right)^{c - 1} \binom{a_n + j - c}j},
\end{split}
\end{equation}

\begin{equation}\tag{B}\label{eq.alkdar8}
\begin{split}
g(x,y;a_n,n,c)&=\sum_{a_{n - 1}  = c}^{a_n } {\sum_{a_{n - 2}  = c}^{a_{n - 1} } { \cdots \sum_{a_0  = c}^{a_1 } {(-1)^{a_0}\left( {\frac{x}{y}} \right)^{a_0 } } } } \\ 
&= (-1)^{a_n}\left( {\frac{x}{{x + y}}} \right)^n \left( {\frac{x}{y}} \right)^{a_n }  + (-1)^c\sum_{j = 0}^{n - 1} {\left( {\frac{x}{{x + y}}} \right)^{n - j}\left(\frac xy\right)^{c - 1} \binom{a_n + j - c}j}.
\end{split}
\end{equation}
Equipped with identities \eqref{eq.rjmjqrv} and \eqref{eq.alkdar8}, we are now ready to state the results regarding the nested sums involving Horadam numbers. Theorems~\ref{thm.zh18syg} and \ref{thm.y4vfanm} are concerned with Fibonacci and Lucas numbers while Theorems~\ref{thm.mrlteua}--\ref{thm.bgdtrf} address the general Horadam sequence.
\begin{theorem}\label{thm.zh18syg}
Let $a_n$, $s$ and $c$ be any integers and let $n$ be a positive integer. Then,
\begin{equation}\tag{F1a}\label{eq.k7jjmd9}
\sum\limits_{a_{n - 1}  = c}^{a_n } {\sum\limits_{a_{n - 2}  = c}^{a_{n - 1} } { \cdots \sum\limits_{a_0  = c}^{a_1 } {F_{3a_0  + s} } } }  = \frac{{F_{2n + 3a_n  + s} }}{{2^n }} - \sum\limits_{j = 0}^{n - 1} {\frac{{F_{2(n - j) + 3(c - 1) + s} }}{{2^{n - j} }}\binom{a_n + j - c}j},
\end{equation}

\begin{equation}\tag{F1b}\label{eq.tzix7ji}
\sum\limits_{a_{n - 1}  = c}^{a_n } {\sum\limits_{a_{n - 2}  = c}^{a_{n - 1} } { \cdots \sum\limits_{a_0  = c}^{a_1 } {L_{3a_0  + s} } } }  = \frac{{L_{2n + 3a_n  + s} }}{{2^n }} - \sum\limits_{j = 0}^{n - 1} {\frac{{L_{2(n - j) + 3(c - 1) + s} }}{{2^{n - j} }}\binom{a_n + j - c}j}.
\end{equation}
\end{theorem}
\begin{proof}
Consider identity \eqref{eq.rjmjqrv}. Simplify both sides of
\[
f(\alpha^3,1;a_n,n,c)\mp f(\beta^3,1;a_n,n,c),
\]
using the Binet formulas~\eqref{eq.qohepwk}.
\end{proof}
\begin{theorem}\label{thm.y4vfanm}
Let $a_n$, $s$ and $c$ be any integers and let $n$ be a positive integer. Then,
\begin{equation}\tag{F2a}\label{eq.g9klvw1}
\sum\limits_{a_{n - 1}  = c}^{a_n } {\sum\limits_{a_{n - 2}  = c}^{a_{n - 1} } { \cdots \sum\limits_{a_0  = c}^{a_1 } {( - 1)^{a_0 } F_{3a_0  + s} } } }  = ( - 1)^{a_n } \frac{{F_{n + 3a_n  + s} }}{{2^n }} + ( - 1)^c \sum\limits_{j = 0}^{n - 1} {\frac{{F_{n - j + 3(c - 1) + s} }}{{2^{n - j} }}\binom{a_n + j - c}j},
\end{equation}

\begin{equation}\tag{F2b}\label{eq.y56t51f}
\sum\limits_{a_{n - 1}  = c}^{a_n } {\sum\limits_{a_{n - 2}  = c}^{a_{n - 1} } { \cdots \sum\limits_{a_0  = c}^{a_1 } {( - 1)^{a_0 } L_{3a_0  + s} } } }  = ( - 1)^{a_n } \frac{{L_{n + 3a_n  + s} }}{{2^n }} + ( - 1)^c \sum\limits_{j = 0}^{n - 1} {\frac{{L_{n - j + 3(c - 1) + s} }}{{2^{n - j} }}\binom{a_n + j - c}j}.
\end{equation}
\end{theorem}
\begin{proof}
Consider identity \eqref{eq.alkdar8}. Simplify both sides of
\[
g(\alpha^3,1;a_n,n,c)\mp g(\beta^3,1;a_n,n,c),
\]
using the Binet formulas~\eqref{eq.qohepwk}.
\end{proof}
\begin{theorem}\label{thm.mrlteua}
Let $r$, $s$, $c$ and $a_n$ be any integers and let $n$ be a positive integer. Then,
\begin{equation}\tag{F3}\label{eq.bouizyw}
\begin{split}
&\sum_{a_{n - 1}  = c}^{a_n } {\sum_{a_{n - 2}  = c}^{a_{n - 1} } { \cdots \sum_{a_0  = c}^{a_1 } {\frac{{W_{ra_0  + s} }}{{V_r^{a_0 } }}} } }\\
&\qquad  = ( - 1)^n \frac{{W_{r(a_n  + 2n) + s} }}{{q^{rn} V_r^{a_n } }} - \frac1{V_r^{c - 1}}\sum_{j = 0}^{n - 1} {( - 1)^{n - j} \frac{{W_{r(2n - 2j + c - 1) + s} }}{{q^{r(n - j)} }}\binom{a_n + j - c}j}.
\end{split}
\end{equation}
 
\end{theorem}
\begin{proof}
Refer to identity \eqref{eq.rjmjqrv}. Use the Binet formulas~\eqref{eq.mt3covv} to simplify both sides of
\[
\mathbb A\tau ^s f(\tau ^r ,V_r ;a_n ,n,c) + \mathbb B\sigma ^s f(\sigma ^r ,V_r ;a_n ,n,c).
\]

\end{proof}
The restricted Horadam sequence version of \eqref{eq.bouizyw} is
\[
\begin{split}
&\sum_{a_{n - 1}  = c}^{a_n } {\sum_{a_{n - 2}  = c}^{a_{n - 1} } { \cdots \sum_{a_0  = c}^{a_1 } {\frac{{w_{ra_0  + s} }}{{v_r^{a_0 } }}} } }\\
&\qquad  = ( - 1)^n \frac{{w_{r(a_n  + 2n) + s} }}{{q^{rn} v_r^{a_n } }} - \frac1{v_r^{c - 1}}\sum_{j = 0}^{n - 1} {( - 1)^{n - j} \frac{{w_{r(2n - 2j + c - 1) + s} }}{{q^{r(n - j)} }}\binom{a_n + j - c}j}.
\end{split}
\]
In particular,
\[
\begin{split}
&\sum_{a_{n - 1}  = c}^{a_n } {\sum_{a_{n - 2}  = c}^{a_{n - 1} } { \cdots \sum_{a_0  = c}^{a_1 } {w_{a_0  + s}} } }\\
&\qquad  = ( - 1)^n \frac{{w_{a_n  + 2n + s} }}{{q^n }} - \sum_{j = 0}^{n - 1} {( - 1)^{n - j} \frac{{w_{2n - 2j + c - 1 + s} }}{{q^{n - j} }}\binom{a_n + j - c}j}.
\end{split}
\]
The gibonacci version of \eqref{eq.bouizyw} is
\[
\begin{split}
&\sum_{a_{n - 1}  = c}^{a_n } {\sum_{a_{n - 2}  = c}^{a_{n - 1} } { \cdots \sum_{a_0  = c}^{a_1 } {\frac{{G_{ra_0  + s} }}{{L_r^{a_0 } }}} } }\\
&\qquad  = ( - 1)^{n( r - 1)} \frac{{G_{r(a_n  + 2n) + s} }}{{L_r^{a_n } }} - \frac1{L_r^{c - 1}}\sum_{j = 0}^{n - 1} {( - 1)^{(n - j)(r - 1)} G_{r(2n - 2j + c - 1) + s}\binom{a_n + j - c}j};
\end{split}
\]
a particular case of which ($r=1$, $s=0$) is
\[
\sum_{a_{n - 1}  = c}^{a_n } {\sum_{a_{n - 2}  = c}^{a_{n - 1} } { \cdots \sum_{a_0  = c}^{a_1 } {G_{a_0 } } } }  = G_{a_n  + 2n}  - \sum_{j = 0}^{n - 1} {G_{2(n - j)} \binom{a_n + j - c}j};
\]
of which \eqref{eq.b1al0dl} is a special case ($c=1$).
\begin{theorem}
Let $r$, $s$, $c$ and $a_n$ be any integers and let $n$ be a positive integer. Then,
\begin{equation}\tag{F4}\label{eq.rnqmj4n}
\begin{split}
\sum_{a_{n - 1}  = c}^{a_n } {\sum_{a_{n - 2}  = c}^{a_{n - 1} } { \cdots \sum_{a_0  = c}^{a_1 } {\frac{{( - 1)^{a_0 } W_{2ra_0  + s} }}{{q^{ra_0 } }}} } }  &= ( - 1)^{a_n } \frac{{W_{r(2a_n  + n) + s} }}{{q^{ra_n } V_r^n }}\\
&\qquad+ \frac{{( - 1)^c }}{{q^{r(c - 1)} }}\sum_{j = 0}^{n - 1} {\frac{{W_{r(n - j + 2c - 2) + s} }}{{V_r^{n - j} }}\binom{a_n + j - c}j}.
\end{split}
\end{equation}
\end{theorem}
\begin{proof}
Refer to~\eqref{eq.alkdar8} and simplify
\[
\mathbb A\tau^sg(\tau^r,\sigma^r;a_n,n,c) + \mathbb B\sigma^sg(\sigma^r,\tau^r;a_n,n,c).
\]
\end{proof}
The gibonacci version of \eqref{eq.rnqmj4n} is
\[
\begin{split}
\sum_{a_{n - 1}  = c}^{a_n } {\sum_{a_{n - 2}  = c}^{a_{n - 1} } { \cdots \sum_{a_0  = c}^{a_1 } {( - 1)^{(r - 1)a_0 } G_{2ra_0  + s} } } }  &= ( - 1)^{(r - 1)a_n } \frac{G_{r(2a_n  + n) + s}}{L_r^n}\\ 
&\qquad + ( - 1)^{r(c - 1) + c} \sum_{j = 0}^{n - 1} {\frac{{G_{r(n - j + 2c - 2) + s} }}{{L_r^{n - j} }}\binom{a_n + j - c}j}. 
\end{split}
\]
In particular,
\[
\sum_{a_{n - 1}  = c}^{a_n } {\sum_{a_{n - 2}  = c}^{a_{n - 1} } { \cdots \sum_{a_0  = c}^{a_1 } {G_{2a_0  + s} } } }  = G_{2a_n  + n + s}  - \sum_{j = 0}^{n - 1} {G_{n - j + 2c - 2 + s} \binom{a_n + j - c}j} .
\]
For the proof of Theorems \ref{thm.l2xlubb} and \ref{thm.xgurb9n}, we require the identities stated in Lemma~\ref{lem.p2d19eb}.
\begin{lemma}\label{lem.p2d19eb}
Let $r$ and $d$ be any integers. Then,
\begin{gather}
U_{r + d}  - \tau ^r U_d  = \sigma ^d U_r\tag{L1}\label{eq.iamiky1},\\ 
U_{r + d}  - \sigma ^r U_d  = \tau ^d U_r\tag{L2}\label{eq.zy0gfyn},\\ 
V_{r + d}  - \tau ^r V_d  =  - \sigma ^d U_r \Delta\tag{L3}\label{eq.j428hfx},\\ 
V_{r + d}  - \sigma ^r V_d  = \tau ^d U_r \Delta\tag{L4}\label{eq.cqli6xc}.
\end{gather}

\end{lemma}
\begin{proof}
Each identity follows directly from the Binet formulas. For example, to prove \eqref{eq.iamiky1}, we have
\[
\begin{split}
U_{r + d}  - \tau ^r U_d & = \frac{{\tau ^{r + d}  - \sigma ^{r + d} }}{{\tau  - \sigma }} - \tau ^r \frac{{\tau ^d  - \sigma ^d }}{{\tau  - \sigma }}\\
 &= \frac{1}{{\tau  - \sigma }}\left( {\tau ^{r + d}  - \sigma ^{r + d}  - \tau ^{r + d}  + \tau ^r \sigma ^d } \right)\\
 &= \frac{{\sigma ^d }}{{\tau  - \sigma }}\left( {\tau ^r  - \sigma ^r } \right) = \sigma ^d U_r .
\end{split}
\]
\end{proof}
\begin{theorem}\label{thm.l2xlubb}
Let $r$, $s$, $c$, $d$ and $a_n$ be any integers; $r\ne 0$, $r + d\ne 0$. Let $n$ be a positive integer. Then,
\begin{equation}\tag{F5}\label{eq.x0drnxk}
\begin{split}
&\sum_{a_{n - 1}  = c}^{a_n } {\sum_{a_{n - 2}  = c}^{a_{n - 1} } { \cdots \sum_{a_0  = c}^{a_1 } {\left( {\frac{{U_d }}{{U_{r + d} }}} \right)^{a_0 } W_{ra_0  + s} } } }\\
&\qquad  = \frac{{( - 1)^n U_d^{n + a_n } }}{{q^{dn} U_r^n U_{r + d}^{a_n } }}W_{(r + d)n + ra_n  + s} \\
&\quad\qquad - \left( {\frac{{U_d }}{{U_{r + d} }}} \right)^{c - 1} \sum_{j = 0}^{n - 1} {\frac{{( - 1)^{n - j} }}{{q^{d(n - j)} }}\left( {\frac{{U_d }}{{U_r }}} \right)^{n - j} W_{r(n - j + c - 1) + d(n - j) + s} \binom{a_n + j - c}j}. 
\end{split}
\end{equation}
\end{theorem}
\begin{proof}
Refer to identity \eqref{eq.rjmjqrv} and simplify both sides of
\[
\mathbb A\tau ^s f(\tau ^r U_d ,U_{r + d} ;a_n ,n) + \mathbb B\sigma ^s f(\sigma ^r U_d ,U_{r + d} ;a_n ,n),
\]
using the Binet formulas and \eqref{eq.iamiky1} and \eqref{eq.zy0gfyn}.
\end{proof}
Note that when $d=r$, \eqref{eq.x0drnxk} reduces to \eqref{eq.bouizyw}.

The gibonacci version of \eqref{eq.x0drnxk} is
\[
\begin{split}
&\sum_{a_{n - 1}  = c}^{a_n } {\sum_{a_{n - 2}  = c}^{a_{n - 1} } { \cdots \sum_{a_0  = c}^{a_1 } {\left( {\frac{{F_d }}{{F_{r + d} }}} \right)^{a_0 } G_{ra_0  + s} } } }\\
&\qquad  = \frac{{( - 1)^{n(d+1)} F_d^{n + a_n } }}{{F_r^n F_{r + d}^{a_n } }}G_{(r + d)n + ra_n  + s} \\
&\quad\qquad - \left( {\frac{{F_d }}{{F_{r + d} }}} \right)^{c - 1} \sum_{j = 0}^{n - 1} {( - 1)^{(n - j)(d + 1)}\left( {\frac{{F_d }}{{F_r }}} \right)^{n - j} G_{r(n - j + c - 1) + d(n - j) + s} \binom{a_n + j - c}j}. 
\end{split}
\]
The identity stated in Lemma~\ref{lem.w65xm59} is required in the proof of Theorem~\ref{thm.xgurb9n}.
\begin{lemma}[{\cite[Lemma 1]{adegoke21}}]\label{lem.w65xm59}
For integer $j$,
\[
\mathbb A\tau ^j  - \mathbb B\sigma ^j  = \frac{{w_{j + 1}  - qw_{j - 1} }}{\Delta }.
\]
\end{lemma}

\begin{theorem}\label{thm.xgurb9n}
Let $r$, $s$, $c$, $d$ and $a_n$ be any integers; $r\ne 0$. If $n$ is a positive even integer, then,
\begin{equation}\tag{F6a}\label{eq.mra1xxh}
\begin{split}
&\sum_{a_{n - 1}  = c}^{a_n } {\sum_{a_{n - 2}  = c}^{a_{n - 1} } { \cdots \sum_{a_0  = c}^{a_1 } {\left( {\frac{{V_d }}{{V_{r + d} }}} \right)^{a_0 } W_{ra_0  + s} } } }\\ 
&\quad = \frac{1}{{q^{dn}\Delta^n }}\left( {\frac{{V_d }}{{U_r}}} \right)^n \left( {\frac{{V_d }}{{V_{r + d} }}} \right)^{a_n } W_{r(n + a_n ) + dn + s}\\ 
&\qquad - \left( {\frac{{V_d }}{{V_{r + d} }}} \right)^{c - 1} \frac1{\Delta^n}\sum_{j = 0}^{(n - 2)/2} {\frac{\Delta^{2j}}{{q^{d(n - 2j)} }}\left( {\frac{{V_d }}{{U_r}}} \right)^{n - 2j} W_{{(r + d)(n - 2j) + r(c - 1) + s} } \binom{a_n + 2j - c}{2j}}\\ 
&\qquad\quad - \left( {\frac{{V_d }}{{V_{r + d} }}} \right)^{c - 1} \frac{1}{\Delta^{n + 2} }\sum_{j = 1}^{n/2} {\left\{ {\frac{\Delta^{2j}}{{q^{d(n - 2j + 1)} }}\left( {\frac{{V_d }}{{U_r}}} \right)^{n - 2j + 1} \left( {W_{{(r + d)(n - 2j + 1) + r(c - 1) + s + 1} }  - } \right.} \right.}\\
&\qquad\qquad\qquad\qquad\qquad\qquad\qquad\left. {\left. {qW_{{(r + d)(n - 2j + 1) + r(c - 1) + s - 1} } } \right)\binom{a_n + 2j - 1 - c}{2j - 1}} \right\}, 
\end{split}
\end{equation}
while if $n$ is a positive odd integer, then,
\begin{equation}\tag{F6b}\label{eq.w1eapv6}
\begin{split}
&\sum_{a_{n - 1}  = c}^{a_n } {\sum_{a_{n - 2}  = c}^{a_{n - 1} } { \cdots \sum_{a_0  = c}^{a_1 } {\left( {\frac{{V_d }}{{V_{r + d} }}} \right)^{a_0 } W_{ra_0  + s} } } }\\ 
&\quad= \frac{1}{{q^{dn} \Delta ^{n + 1} }}\left( {\frac{{V_d }}{{U_r }}} \right)^n \left( {\frac{{V_d }}{{V_{r + d} }}} \right)^{a_n } \left( {W_{{r(n + a_n ) + dn + s + 1} }  - qW_{{r(n + a_n ) + dn + s - 1} } } \right)\\
&\qquad- \left( {\frac{{V_d }}{{V_{r + d} }}} \right)^{c - 1} \frac{1}{{\Delta ^{n + 1} }}\sum_{j = 0}^{(n - 1)/2} {\left\{ {\frac{{\Delta ^{2j} }}{{q^{d(n - 2j)} }}\left( {\frac{{V_d }}{{U_r }}} \right)^{n - 2j} \left( {W_{(r + d)(n - 2j) + r(c - 1) + s + 1}  - } \right.} \right.} \\
&\qquad\qquad\qquad\qquad\qquad\qquad\qquad\qquad\left. {\left. {qW_{(r + d)(n - 2j) + r(c - 1) + s - 1} } \right)\binom{a_n + 2j - c}{2j}} \right\}\\
&\quad\qquad- \left( {\frac{{V_d }}{{V_{r + d} }}} \right)^{c - 1} \frac{1}{{\Delta ^{n + 1} }}\sum_{j = 1}^{(n - 1)/2 } {\frac{{\Delta ^{2j} }}{{q^{d(n - 2j + 1)} }}\left( {\frac{{V_d }}{{U_r }}} \right)^{n - 2j + 1} W_{{(r + d)(n - 2j + 1) + r(c - 1) + s} } \binom{a_n + 2j - 1 - c}{2j - 1}} .
\end{split}
\end{equation}
\end{theorem}
\begin{proof}
First employ the summation identity
\[
\sum_{j = 0}^m {f_j }  = \sum_{j = 0}^{\left\lfloor {m/2} \right\rfloor } {f_{2j} }  + \sum_{j = 1}^{\left\lceil {m/2} \right\rceil } {f_{2j - 1} } 
\]
to write \eqref{eq.rjmjqrv} as
\begin{equation}\tag{A2}\label{eq.e0mrmmw}
\begin{split}
f(x,y;a_n ,n,c) &= \sum_{a_{n - 1}  = c}^{a_n } {\sum_{a_{n - 2}  = c}^{a_{n - 1} } { \cdots \sum_{a_0  = c}^{a_1 } {\left( {\frac{x}{y}} \right)^{a_0 } } } } \\
&= \left( {\frac{x}{{x - y}}} \right)^n \left( {\frac{x}{y}} \right)^{a_n }  - \sum_{j = 0}^{\left\lfloor {(n - 1)/2} \right\rfloor } {\left( {\frac{x}{{x - y}}} \right)^{n - 2j} \left( {\frac{x}{y}} \right)^{c - 1} \binom{a_n + 2j - c}{2j}}\\ 
&\qquad- \sum_{j = 1}^{\left\lceil {(n - 1)/2} \right\rceil } {\left( {\frac{x}{{x - y}}} \right)^{n - 2j + 1} \left( {\frac{x}{y}} \right)^{c - 1} \binom{a_n + 2j - 1 - c}{2j - 1}}. 
\end{split}
\end{equation}
Using \eqref{eq.e0mrmmw}, the Binet formulas and \eqref{eq.j428hfx} and \eqref{eq.cqli6xc},
\[
\mathbb A\tau ^s f(\tau ^r V_d ,V_{r + d} ;a_n ,n) + \mathbb B\sigma ^s f(\sigma ^r V_d ,V_{r + d} ;a_n ,n)
\]
evaluates to
\begin{equation}\label{eq.qam9xta}
\begin{split}
&\sum_{a_{n - 1}  = c}^{a_n } {\sum_{a_{n - 2}  = c}^{a_{n - 1} } { \cdots \sum_{a_0  = c}^{a_1 } {\left( {\frac{{V_d }}{{V_{r + d} }}} \right)^{a_0 } W_{ra_0  + s} } } }\\ 
&= \frac{1}{{q^{dn} }}\left( {\frac{{V_d }}{{U_r \Delta }}} \right)^n \left( {\frac{{V_d }}{{V_{r + d} }}} \right)^{a_n } \left( {\mathbb A\tau ^{r(n + a_n ) + dn + s}  + ( - 1)^n \mathbb B\sigma ^{r(n + a_n ) + dn + s} } \right)\\
&\quad - \left( {\frac{{V_d }}{{V_{r + d} }}} \right)^{c - 1} \sum_{j = 0}^{\left\lfloor {(n - 1)/2} \right\rfloor } {\left\{ {\frac{1}{{q^{d(n - 2j)} }}\left( {\frac{{V_d }}{{U_r \Delta }}} \right)^{n - 2j} \left( {\mathbb A\tau ^{(r + d)(n - 2j) + r(c - 1) + s}  + } \right.} \right.} \\
&\qquad\qquad\qquad\qquad\qquad\qquad\left. {\left. {( - 1)^n \mathbb B\sigma ^{(r + d)(n - 2j) + r(c - 1) + s} } \right)\binom{a_n + 2j - c}{2j}} \right\}\\
&\quad\qquad- \left( {\frac{{V_d }}{{V_{r + d} }}} \right)^{c - 1} \sum_{j = 1}^{\left\lceil {(n - 1)/2} \right\rceil } {\left\{ {\frac{1}{{q^{d(n - 2j + 1)} }}\left( {\frac{{V_d }}{{U_r \Delta }}} \right)^{n - 2j + 1} \left( {\mathbb A\tau ^{(r + d)(n - 2j + 1) + r(c - 1) + s}  - } \right.} \right.}\\ 
&\qquad\qquad\qquad\qquad\qquad\qquad\qquad\left. {\left. {( - 1)^n \mathbb B\sigma ^{(r + d)(n - 2j + 1) + r(c - 1) + s} } \right)\binom{a_n + 2j - 1 - c}{2j - 1}} \right\}.
\end{split}
\end{equation}
The stated results now follow from the parity consideration of $n$ in \eqref{eq.qam9xta}, the Binet formulas and Lemma~\ref{lem.w65xm59}.
\end{proof}
We now state the gibonacci versions of \eqref{eq.mra1xxh} and \eqref{eq.w1eapv6} for any integers $r$, $s$, $c$, $d$ and $a_n$ such that $r\ne 0$. If $n$ is a positive even integer, then,
\begin{equation*}
\begin{split}
&\sum_{a_{n - 1}  = c}^{a_n } {\sum_{a_{n - 2}  = c}^{a_{n - 1} } { \cdots \sum_{a_0  = c}^{a_1 } {\left( {\frac{{L_d }}{{L_{r + d} }}} \right)^{a_0 } G_{ra_0  + s} } } }\\ 
&\quad = \frac{1}{{5^{n/2} }}\left( {\frac{{L_d }}{{F_r}}} \right)^n \left( {\frac{{L_d }}{{L_{r + d} }}} \right)^{a_n } G_{r(n + a_n ) + dn + s}\\ 
&\qquad - \left( {\frac{{L_d }}{{L_{r + d} }}} \right)^{c - 1} \frac1{5^{n/2}}\sum_{j = 0}^{(n - 2)/2} {5^j\left( {\frac{{L_d }}{{F_r}}} \right)^{n - 2j} G_{{(r + d)(n - 2j) + r(c - 1) + s} } \binom{a_n + 2j - c}{2j}}\\ 
&\qquad\quad - \left( {\frac{{L_d }}{{L_{r + d} }}} \right)^{c - 1} \frac{(-1)^d}{5^{(n + 2)/2} }\sum_{j = 1}^{n/2} {\left\{ {5^j\left( {\frac{{L_d }}{{F_r}}} \right)^{n - 2j + 1} \left( {G_{{(r + d)(n - 2j + 1) + r(c - 1) + s + 1} }  + } \right.} \right.}\\
&\qquad\qquad\qquad\qquad\qquad\qquad\qquad\qquad\left. {\left. {G_{{(r + d)(n - 2j + 1) + r(c - 1) + s - 1} } } \right)\binom{a_n + 2j - 1 - c}{2j - 1}} \right\}, 
\end{split}
\end{equation*}
while if $n$ is a positive odd integer, then,
\begin{equation*}
\begin{split}
&\sum_{a_{n - 1}  = c}^{a_n } {\sum_{a_{n - 2}  = c}^{a_{n - 1} } { \cdots \sum_{a_0  = c}^{a_1 } {\left( {\frac{{L_d }}{{L_{r + d} }}} \right)^{a_0 } G_{ra_0  + s} } } }\\ 
&\quad= \frac{(-1)^d}{{5 ^{(n + 1)/2} }}\left( {\frac{{L_d }}{{F_r }}} \right)^n \left( {\frac{{L_d }}{{L_{r + d} }}} \right)^{a_n } \left( {G_{{r(n + a_n ) + dn + s + 1} }  + G_{{r(n + a_n ) + dn + s - 1} } } \right)\\
&\qquad- \left( {\frac{{L_d }}{{L_{r + d} }}} \right)^{c - 1} \frac{(-1)^d}{{5 ^{(n + 1)/2} }}\sum_{j = 0}^{(n - 1)/2} {\left\{ {5^j\left( {\frac{{L_d }}{{F_r }}} \right)^{n - 2j} \left( {G_{(r + d)(n - 2j) + r(c - 1) + s + 1}  + } \right.} \right.} \\
&\qquad\qquad\qquad\qquad\qquad\qquad\qquad\qquad\left. {\left. {G_{(r + d)(n - 2j) + r(c - 1) + s - 1} } \right)\binom{a_n + 2j - c}{2j}} \right\}\\
&\quad\qquad- \left( {\frac{{L_d }}{{L_{r + d} }}} \right)^{c - 1} \frac{1}{{5 ^{(n + 1)/2} }}\sum_{j = 1}^{(n - 1)/2 } {5^j\left( {\frac{{L_d }}{{F_r }}} \right)^{n - 2j + 1} G_{{(r + d)(n - 2j + 1) + r(c - 1) + s} } \binom{a_n + 2j - 1 - c}{2j - 1}} .
\end{split}
\end{equation*}
In particular, for the Fibonacci sequence, we have that if $n$ is a positive even integer, then,
\begin{equation*}
\begin{split}
&\sum_{a_{n - 1}  = c}^{a_n } {\sum_{a_{n - 2}  = c}^{a_{n - 1} } { \cdots \sum_{a_0  = c}^{a_1 } {\left( {\frac{{L_d }}{{L_{r + d} }}} \right)^{a_0 } F_{ra_0  + s} } } }\\ 
&\quad = \frac{1}{{5^{n/2} }}\left( {\frac{{L_d }}{{F_r}}} \right)^n \left( {\frac{{L_d }}{{L_{r + d} }}} \right)^{a_n } F_{r(n + a_n ) + dn + s}\\ 
&\qquad - \left( {\frac{{L_d }}{{L_{r + d} }}} \right)^{c - 1} \frac1{5^{n/2}}\sum_{j = 0}^{(n - 2)/2} {5^j\left( {\frac{{L_d }}{{F_r}}} \right)^{n - 2j} F_{{(r + d)(n - 2j) + r(c - 1) + s} } \binom{a_n + 2j - c}{2j}}\\ 
&\qquad\quad - \left( {\frac{{L_d }}{{L_{r + d} }}} \right)^{c - 1} \frac{(-1)^d}{5^{(n + 2)/2} }\sum_{j = 1}^{n/2} {{5^j\left( {\frac{{L_d }}{{F_r}}} \right)^{n - 2j + 1}L_{{(r + d)(n - 2j + 1) + r(c - 1) + s} } \binom{a_n + 2j - 1 - c}{2j - 1}}}, 
\end{split}
\end{equation*}
while if $n$ is a positive odd integer, then,
\begin{equation*}
\begin{split}
&\sum_{a_{n - 1}  = c}^{a_n } {\sum_{a_{n - 2}  = c}^{a_{n - 1} } { \cdots \sum_{a_0  = c}^{a_1 } {\left( {\frac{{L_d }}{{L_{r + d} }}} \right)^{a_0 } F_{ra_0  + s} } } }\\ 
&\quad= \frac{(-1)^d}{{5 ^{(n + 1)/2} }}\left( {\frac{{L_d }}{{F_r }}} \right)^n \left( {\frac{{L_d }}{{L_{r + d} }}} \right)^{a_n } {L_{{r(n + a_n ) + dn + s} } }\\
&\qquad- \left( {\frac{{L_d }}{{L_{r + d} }}} \right)^{c - 1} \frac{(-1)^d}{{5 ^{(n + 1)/2} }}\sum_{j = 0}^{(n - 1)/2} {{5^j\left( {\frac{{L_d }}{{F_r }}} \right)^{n - 2j}L_{(r + d)(n - 2j) + r(c - 1) + s}\binom{a_n + 2j - c}{2j}}}\\
&\quad\qquad- \left( {\frac{{L_d }}{{L_{r + d} }}} \right)^{c - 1} \frac{1}{{5 ^{(n + 1)/2} }}\sum_{j = 1}^{(n - 1)/2 } {5^j\left( {\frac{{L_d }}{{F_r }}} \right)^{n - 2j + 1} F_{{(r + d)(n - 2j + 1) + r(c - 1) + s} } \binom{a_n + 2j - 1 - c}{2j - 1}} .
\end{split}
\end{equation*}
As for the sequence of Lucas numbers, we have that if $n$ is a positive even integer, then,
\begin{equation*}
\begin{split}
&\sum_{a_{n - 1}  = c}^{a_n } {\sum_{a_{n - 2}  = c}^{a_{n - 1} } { \cdots \sum_{a_0  = c}^{a_1 } {\left( {\frac{{L_d }}{{L_{r + d} }}} \right)^{a_0 } L_{ra_0  + s} } } }\\ 
&\quad = \frac{1}{{5^{n/2} }}\left( {\frac{{L_d }}{{F_r}}} \right)^n \left( {\frac{{L_d }}{{L_{r + d} }}} \right)^{a_n } L_{r(n + a_n ) + dn + s}\\ 
&\qquad - \left( {\frac{{L_d }}{{L_{r + d} }}} \right)^{c - 1} \frac1{5^{n/2}}\sum_{j = 0}^{(n - 2)/2} {5^j\left( {\frac{{L_d }}{{F_r}}} \right)^{n - 2j} L_{{(r + d)(n - 2j) + r(c - 1) + s} } \binom{a_n + 2j - c}{2j}}\\ 
&\qquad\quad - \left( {\frac{{L_d }}{{L_{r + d} }}} \right)^{c - 1} \frac{(-1)^d}{5^{n/2} }\sum_{j = 1}^{n/2} { {5^j\left( {\frac{{L_d }}{{F_r}}} \right)^{n - 2j + 1}{F_{{(r + d)(n - 2j + 1) + r(c - 1) + s} }}\binom{a_n + 2j - 1 - c}{2j - 1}}}, 
\end{split}
\end{equation*}
while if $n$ is a positive odd integer, then,
\begin{equation*}
\begin{split}
&\sum_{a_{n - 1}  = c}^{a_n } {\sum_{a_{n - 2}  = c}^{a_{n - 1} } { \cdots \sum_{a_0  = c}^{a_1 } {\left( {\frac{{L_d }}{{L_{r + d} }}} \right)^{a_0 } L_{ra_0  + s} } } }\\ 
&\quad= \frac{(-1)^d}{{5 ^{(n - 1)/2} }}\left( {\frac{{L_d }}{{F_r }}} \right)^n \left( {\frac{{L_d }}{{L_{r + d} }}} \right)^{a_n } {F_{{r(n + a_n ) + dn + s} }}\\
&\qquad- \left( {\frac{{L_d }}{{L_{r + d} }}} \right)^{c - 1} \frac{(-1)^d}{{5 ^{(n - 1)/2} }}\sum_{j = 0}^{(n - 1)/2} { {5^j\left( {\frac{{L_d }}{{F_r }}} \right)^{n - 2j} {F_{(r + d)(n - 2j) + r(c - 1) + s}}\binom{a_n + 2j - c}{2j}} }\\
&\quad\qquad- \left( {\frac{{L_d }}{{L_{r + d} }}} \right)^{c - 1} \frac{1}{{5 ^{(n + 1)/2} }}\sum_{j = 1}^{(n - 1)/2 } {5^j\left( {\frac{{L_d }}{{F_r }}} \right)^{n - 2j + 1} L_{{(r + d)(n - 2j + 1) + r(c - 1) + s} } \binom{a_n + 2j - 1 - c}{2j - 1}} .
\end{split}
\end{equation*}
\begin{theorem}\label{thm.bgdtrf}
Let $r$, $s$, $d$, $a_n$, $c$ be any integers; $r + 1\ne d$. If $W_{r+s}\ne 0$ and $W_{r+d}\ne 0$ then,
\begin{equation}\tag{F7}\label{eq.xu0xef5}
\begin{split}
&\sum\limits_{a_{n - 1}  = c}^{a_n } {\sum\limits_{a_{n - 2}  = c}^{a_{n - 1} } { \cdots \sum\limits_{a_0  = c}^{a_1 } {q^{a_0 } \left( {\frac{{U_{r - d} }}{{U_{r - d + 1} }}} \right)^{a_0 } \left( {\frac{{W_{s + d - 1} }}{{W_{s + d} }}} \right)^{a_0 } } } } \\
 &\qquad= ( - 1)^n q^{n + a_n } U_{r - d}^n \left( {\frac{{U_{r - d} }}{{U_{r - d + 1} }}} \right)^{a_n } \left( {\frac{{W_{s + d - 1} }}{{W_{s + d} }}} \right)^{a_n } \left( {\frac{{W_{s + d - 1} }}{{W_{r + s} }}} \right)^n \\
&\qquad\qquad - q^{c - 1} \left( {\frac{{U_{r - d} }}{{U_{r - d + 1} }}} \right)^{c - 1} \left( {\frac{{W_{s + d - 1} }}{{W_{s + d} }}} \right)^{c - 1} \sum\limits_{j = 0}^{n - 1} {( - 1)^{n - j} q^{n - j} U_{r - d}^{n - j} \left( {\frac{{W_{s + d - 1} }}{{W_{r + s} }}} \right)^{n - j} \binom{a_n + j - c}j} .
\end{split}
\end{equation}
\end{theorem}
\begin{proof}
Refer to~\eqref{eq.rjmjqrv} and simplify
\[
f(qU_{r - d}W_{s + d - 1},U_{r - d + 1}W_{s + d};a_n,n,c),
\]
making use of the following identity~\cite[Identity 3.15]{horadam65}:
\[
W_{r + s}  = U_{r - d + 1} W_{s + d}  - qU_{r - d} W_{s + d - 1}. 
\]

\end{proof}
The restricted Horadam version ($p=1$) and the gibonacci version of \eqref{eq.xu0xef5} are
\[
\begin{split}
&\sum\limits_{a_{n - 1}  = c}^{a_n } {\sum\limits_{a_{n - 2}  = c}^{a_{n - 1} } { \cdots \sum\limits_{a_0  = c}^{a_1 } {q^{a_0 } \left( {\frac{{u_{r - d} }}{{u_{r - d + 1} }}} \right)^{a_0 } \left( {\frac{{w_{s + d - 1} }}{{w_{s + d} }}} \right)^{a_0 } } } } \\
 &\qquad= ( - 1)^n q^{n + a_n } u_{r - d}^n \left( {\frac{{u_{r - d} }}{{u_{r - d + 1} }}} \right)^{a_n } \left( {\frac{{w_{s + d - 1} }}{{w_{s + d} }}} \right)^{a_n } \left( {\frac{{w_{s + d - 1} }}{{w_{r + s} }}} \right)^n \\
&\qquad\qquad - q^{c - 1} \left( {\frac{{u_{r - d} }}{{u_{r - d + 1} }}} \right)^{c - 1} \left( {\frac{{w_{s + d - 1} }}{{w_{s + d} }}} \right)^{c - 1} \sum\limits_{j = 0}^{n - 1} {( - 1)^{n - j} q^{n - j} u_{r - d}^{n - j} \left( {\frac{{w_{s + d - 1} }}{{w_{r + s} }}} \right)^{n - j} \binom{a_n + j - c}j},
\end{split}
\]

\[
\begin{split}
&\sum\limits_{a_{n - 1}  = c}^{a_n } {\sum\limits_{a_{n - 2}  = c}^{a_{n - 1} } { \cdots \sum\limits_{a_0  = c}^{a_1 } {(-1)^{a_0 } \left( {\frac{{F_{r - d} }}{{F_{r - d + 1} }}} \right)^{a_0 } \left( {\frac{{G_{s + d - 1} }}{{G_{s + d} }}} \right)^{a_0 } } } } \\
 &\qquad= ( - 1)^{a_n} F_{r - d}^n \left( {\frac{{F_{r - d} }}{{F_{r - d + 1} }}} \right)^{a_n } \left( {\frac{{G_{s + d - 1} }}{{G_{s + d} }}} \right)^{a_n } \left( {\frac{{G_{s + d - 1} }}{{G_{r + s} }}} \right)^n \\
&\qquad\qquad + (-1)^c \left( {\frac{{F_{r - d} }}{{F_{r - d + 1} }}} \right)^{c - 1} \left( {\frac{{G_{s + d - 1} }}{{G_{s + d} }}} \right)^{c - 1} \sum\limits_{j = 0}^{n - 1} {F_{r - d}^{n - j} \left( {\frac{{G_{s + d - 1} }}{{G_{r + s} }}} \right)^{n - j} \binom{a_n + j - c}j} .
\end{split}
\]
Taking $r=1$, $d=0$ gives
\[
\begin{split}
&\sum\limits_{a_{n - 1}  = c}^{a_n } {\sum\limits_{a_{n - 2}  = c}^{a_{n - 1} } { \cdots \sum\limits_{a_0  = c}^{a_1 } {q^{a_0 } \left( {\frac{{w_{s - 1} }}{{w_s }}} \right)^{a_0 } } } } \\
 &\qquad= ( - 1)^n q^{n + a_n } \left( {\frac{{w_{s - 1} }}{{w_s }}} \right)^{a_n } \left( {\frac{{w_{s - 1} }}{{w_{s + 1} }}} \right)^n \\
&\qquad\qquad - q^{c - 1} \left( {\frac{{w_{s - 1} }}{{w_s }}} \right)^{c - 1} \sum\limits_{j = 0}^{n - 1} {( - 1)^{n - j} q^{n - j} \left( {\frac{{w_{s - 1} }}{{w_{s + 1} }}} \right)^{n - j} \binom{a_n + j - c}j}, 
\end{split}
\]

\[
\begin{split}
&\sum\limits_{a_{n - 1}  = c}^{a_n } {\sum\limits_{a_{n - 2}  = c}^{a_{n - 1} } { \cdots \sum\limits_{a_0  = c}^{a_1 } {(-1)^{a_0 } \left( {\frac{{G_{s - 1} }}{{G_s }}} \right)^{a_0 } } } } \\
 &\qquad= ( - 1)^{a_n} \left( {\frac{{G_{s - 1} }}{{G_s }}} \right)^{a_n } \left( {\frac{{G_{s - 1} }}{{G_{s + 1} }}} \right)^n \\
&\qquad\qquad + (-1)^c \left( {\frac{{G_{s - 1} }}{{G_s }}} \right)^{c - 1} \sum\limits_{j = 0}^{n - 1} {\left( {\frac{{G_{s - 1} }}{{G_{s + 1} }}} \right)^{n - j} \binom{a_n + j - c}j}. 
\end{split}
\]
\section{Concluding comments}
In this paper we evaluated the following nested sums involving the terms of the Horadam sequence:
\[
\begin{split}
&\sum_{a_{n - 1}  = c}^{a_n } {\sum_{a_{n - 2}  = c}^{a_{n - 1} } { \cdots \sum_{a_0  = c}^{a_1 } {\frac{{W_{ra_0  + s} }}{{V_r^{a_0 } }}} } },\quad \sum_{a_{n - 1}  = c}^{a_n } {\sum_{a_{n - 2}  = c}^{a_{n - 1} } { \cdots \sum_{a_0  = c}^{a_1 } {\frac{{( - 1)^{a_0 } W_{2ra_0  + s} }}{{q^{ra_0 } }}} } } ,\\
&\\
&\sum_{a_{n - 1}  = c}^{a_n } {\sum_{a_{n - 2}  = c}^{a_{n - 1} } { \cdots \sum_{a_0  = c}^{a_1 } {\left( {\frac{{U_d }}{{U_{r + d} }}} \right)^{a_0 } W_{ra_0  + s} } } },\quad\sum_{a_{n - 1}  = c}^{a_n } {\sum_{a_{n - 2}  = c}^{a_{n - 1} } { \cdots \sum_{a_0  = c}^{a_1 } {\left( {\frac{{V_d }}{{V_{r + d} }}} \right)^{a_0 } W_{ra_0  + s} } } },\\
&\\
&\sum\limits_{a_{n - 1}  = c}^{a_n } {\sum\limits_{a_{n - 2}  = c}^{a_{n - 1} } { \cdots \sum\limits_{a_0  = c}^{a_1 } {q^{a_0 } \left( {\frac{{U_{r - d} }}{{U_{r - d + 1} }}} \right)^{a_0 } \left( {\frac{{W_{s + d - 1} }}{{W_{s + d} }}} \right)^{a_0 } } } }.
\end{split}
\]
Explicit results for the special Horadam sequence $(w_j)$ and the gibonacci sequence $G_j$ were presented.

The starting point for each summation in the multiple sums considered was fixed at $a_i=c$, for $i=0,1,2,\ldots,n - 1$. It is possible to choose a different lower limit for each sum in the nested sum. The aim would then be to evaluate
\begin{equation}\tag{S}\label{eq.ea960e8}
\sum_{a_{n - 1}  = c_{n - 1} }^{a_n } {\sum_{a_{n - 2}  = c_{n - 2} }^{a_{n - 1} } { \cdots \sum_{a_1  = c_1 }^{a_2 } {\sum_{a_0  = c_0 }^{a_1 } {x^{a_0 } } } } },
\end{equation}
in which the sequence of integers $(c_i)$, $i=0,1,2,\ldots,n - 1$, is such that $c_i$ may be different from $c_j$ if $i$ is different from $j$. In this case, a similar iteration to what produced \eqref{eq.v0zquxz} would give
\[
\begin{split}
&\left( {\frac{{x - 1}}{x}} \right)^n \sum_{a_{n - 1}  = c_{n - 1} }^{a_n } {\sum_{a_{n - 2}  = c_{n - 2} }^{a_{n - 1} } { \cdots \sum_{a_0  = c_0 }^{a_1 } {x^{a_0 } } } }\\ 
&\qquad= x^{a_n }  - x^{c_{n - 1}  - 1}  - \sum_{j = 1}^{n - 1} {\left\{ {\left( {\frac{{x - 1}}{x}} \right)^j x^{c_{n - j - 1}  - 1} \sum_{a_{n - 1}  = c_{n - 1} }^{a_n } {\sum_{a_{n - 2}  = c_{n - 2} }^{a_{n - 1} } { \cdots \sum_{a_{n - j}  = c_{n - j} }^{a_{n - j + 1} } 1 } } } \right\}}. 
\end{split}
\]
A suggestion for further research would be the determination of sums of the form
\[
\sum_{b_{s - 1}  = c_{s - 1} }^{b_s } {\sum_{b_{s - 2}  = c_{s - 2} }^{b_{s - 1} } { \cdots \sum_{b_1  = c_1 }^{b_2 } {\sum_{b_0  = c_0 }^{b_1 } {1} } } },
\]
which would facilitate the evaluation of \eqref{eq.ea960e8}.

\hrule



\hrule



\end{document}